\documentclass[twocolumn,a4paper]{autart}
\usepackage{InvKalmanFilterAutomatica}
\usetikzlibrary{spy,backgrounds}
\usepackage{pgfplotstable}
\renewcommand{\thefootnote}{\fnsymbol{footnote}}
\allowdisplaybreaks

\begin{document}

\setlength{\abovedisplayskip}{3pt}
\setlength{\belowdisplayskip}{3pt}

\begin{frontmatter}

\title{Invariant extended Kalman filter on matrix Lie groups} 
\author[authors]{Karmvir Singh Phogat}\ead{karmvir.p@gmail.com},
\author[authors]{Dong Eui Chang}\ead{dechang@kaist.ac.kr}
\address[authors]{School of Electrical Engineering, KAIST, Daejeon, 34141, South Korea}
          
\begin{keyword}
	Extended Kalman filter, Lie groups, differential Riccati equation.
\end{keyword}
                             
\begin{abstract}
We derive symmetry preserving invariant extended Kalman filters (IEKF) on matrix Lie groups. These Kalman filters have an advantage over conventional extended Kalman filters as the error dynamics for such filters are independent of the group configuration which, in turn, provides a uniform estimate of the region of convergence. The proposed IEKF differs from existing techniques in literature on the account that it is derived using minimal tools from differential geometry that simplifies its representation and derivation to a large extent. The filter error dynamics is defined on the Lie algebra directly instead of identifying the Lie algebra with an Euclidean space or defining the error dynamics in local coordinates using exponential map, and the associated differential Riccati equations are described on the corresponding space of linear operators using tensor algebra. The proposed filter is implemented for the attitude dynamics of the rigid body, which is a benchmark problem in control, and its performance is compared against a conventional extended Kalman filter (EKF). Numerical experiments support that the IEKF is computationally less intensive and gives better performance than the EKF.  
\end{abstract}

\end{frontmatter}

\section{Introduction}
The key objective of filtering is to estimate the state of a dynamical system by incorporating the state evolution driven by modeled system dynamics and the partial state information received from sensor measurements. The sensor measurements are corrupted with noise, and on the other hand, the system models are inherently inaccurate. Therefore, a filtering algorithm that provides an unbiased state estimation with minimum error variance is most desirable. One of the successful attempts in filter design for linear systems is the Kalman filter \cite{kalman1961} which is extensively used in various domains ranging from robotics, autonomous systems, finance and biology. The Kalman filtering technique is further extended to nonlinear systems and the resulting filters are popularly known, in literature, as extended Kalman filters (EKF) \cite{reif2000}. In contrast to classical Kalman filters, extended Kalman filters do not guarantee optimality, and may diverge, in case of fast dynamics, if the initial state estimates are not accurate enough. A detailed discussion on the convergence of EKF may be found in \cite{reif2000,krener2003}.

	For mechanical systems evolving on Lie groups with invariant dynamics, due to symmetry, a special class of symmetry preserving Kalman filters are devised which ensures that the error dynamics is independent of the group configuration. These filters are known in research literature as invariant extended Kalman filters (IEKFs) \cite{bonnabel2007,saccon2016,barrau2017,bourmaud2015,bonnabel2009}. The IEKF was first proposed for continuous-time systems on Lie groups with output measurements in Euclidean spaces in \cite{bonnabel2007} and for systems defined on open sets in Euclidean spaces having Lie group symmetry in \cite{bonnabel2009}. Subsequently, IEKFs have been generalized to continuous-time systems on Lie groups with discrete-time measurements in Euclidean spaces in \cite{barrau2015,barrau2017} and for discrete-time measurements on Lie groups in \cite{bourmaud2015}. A successful implementation of the IEKF for optical-inertial tracking system may be found in \cite{claasen2011}. In contrast to these techniques that employ group symmetry, a continuous-time Kalman filter on the orthogonal group is derived by posing a constrained optimal control problem that preserves the manifold structure in \cite{de2017}. The central idea behind IEKFs for systems on Lie groups, in existing literature, is to translate the error dynamics to Euclidean spaces using the \textit{log} map. Subsequently, the translated error dynamics is linearized and the Kalman gains are found by minimizing the variance of the linearized error dynamics. This intriguing idea provides a natural generalization of the extended Kalman filters to Lie groups. However, for invariant systems on Lie groups, the translation of the error dynamics to Euclidean spaces and its linearization can be largely simplified by directly deriving the linear error dynamics on the Lie algebra from the nonlinear error dynamics using variational approaches.
        
	In this article, we derive continuous-time IEKFs for invariant systems with measurements on matrix Lie groups. We provide a concise derivation of the IEKF in which linearization of the nonlinear error dynamics is done using variational approaches. This step avoids translation of the error dynamics to Euclidean spaces, using \textit{log} map, before its linearization. Therefore, the proposed IEKF is computationally less intensive as it does not require linearization of the error dynamics in local charts \cite{barrau2017,barrau2018invariant,bonnabel2007,bonnabel2009} at each iteration. Another distinguishing feature of the proposed technique is that the Kalman gain and the associated differential Riccati equation (DRE) are derived directly in operator spaces using tensor algebra. Moreover, to validate the performance of the proposed IEKF, we design the IEKF for the rigid body attitude dynamics and compare it with the conventional EKF. It is concluded from the numerical experiments that the proposed IEKF is computationally efficient and performs better than the conventional EKF. 

	This article unfolds as follows: Section \ref{sec:LeftInv} provides an introduction to left invariant systems on Lie groups, and the IEKF is proposed in Sec. \ref{sec:IEKF}. Section \ref{sec:DRE} is devoted to derivation of the Kalman gains and the associated differential Riccati equation in operator spaces. The implementation of the IEKF for the rigid body attitude dynamics and a comparative study of the IEKF and the EKF are given in Sec. \ref{sec:RigidBody}.         

\section{Left invariant systems on Lie groups}\label{sec:LeftInv}
	For ease of presentation, let us first define notions related to Lie groups:
\begin{definition}[{{Lie group}}]
A manifold $G$ is a Lie group if $G$ is a group, in the algebraic sense, with the property that the multiplication map 
\[
G \times G \ni (g,h) \mapsto m(g,h) \Let gh \in G,
\] 
and the inversion map 
\[
G  \ni g \mapsto i(g) \Let g^{-1} \in G
\] 
are smooth.
\end{definition}
Note that every Lie group is associated with its unique Lie algebra, and the Lie algebra $\mathfrak{g}$ of a Lie group $G$ is defined by the tangent space of $G$ at the group identity $I$, i.e., 
\[
\mathfrak{g} \Let T_{I}G.
\]
\begin{definition}[{{Left-action}}]
A left action of a Lie group $G$ on a manifold $M$ is a map 
\[G \times M \ni (g,p) \mapsto L_g(p) \Let g p \in M, \]
that satisfies
\begin{enumerate}
\item \(g_1 (g_2 p) = (g_1 g_2) p \quad \text{for all } g_1, g_2 \in G \;\text{ and } p \in M;\)
\item \(I p = p \quad \text{for all } p \in M, \; \text{ and the group identity, } I \in G. \)
\end{enumerate}
\end{definition}
\begin{definition}[{{Left-invariance}}]
A vector field $\mathcal{X}$ on the manifold $M$ is said to be left invariant if 
\[
T_{p} L_g\left(\mathcal{X}(p)\right) = \mathcal{X}\left(L_g(p)\right) \quad \text{ for all } g \in G \text{ and } p \in M,
\] 
where $T_{p} L_g$ is the tangent lift of the left action $L_g$ at $p \in M$.
A map $M \ni p \mapsto f(p) \in M$ is said to be left invariant under the left action if  
\[
f\left(gp\right) = f(p) \quad \text{ for all } g \in G \; \text{ and } p \in M.
\]
\end{definition}
A detailed exposition of Lie groups may be found in \cite{barrau2018invariant,lee}.

	 Let \m{G} be a matrix Lie group of dimension \m{m} with \m{I} as the group identity, and \m{\mathfrak{g}} be the Lie algebra of the Lie group \m{G}. Suppose that the matrix Lie group \m{G} is embedded into \m{\R^{n\times n}} and there is a projection map
\[
\R^{n \times n} \ni \alpha \mapsto \pi_{\mathfrak{g}}(\alpha) \in \mathfrak{g}.
\]
To fix notation, take an inner product $\langle\, , \rangle$  on $\R^{n\times n}$, and equip $\mathfrak{g}$ with the same inner product by restriction, and assume that the inner product on subspaces of \m{\R^{n\times n}} and their tensor or product spaces are induced from $\R^{n\times n}.$ Furthermore, we denote 
\[
 \mathfrak{g}^{k} \Let \underbrace{\mathfrak{g} \times \cdots \times\mathfrak{g}}_{k \text{ times}} \quad \text{ for } k \in \N,
\]
$L(\mathfrak{g},\mathfrak{g})$ as the space of linear operators from $\mathfrak{g}$ to $\mathfrak{g}$, and the adjoint operator of a linear map \m{\mathcal{A}} is denoted by $\mathcal{A}^*$. 

	Consider a left invariant system dynamics\footnote{Invariant systems on Lie groups accounts for a rich class of control systems with application in various domains such as robotics \cite{chirikjian2010information}, aerospace \cite{barrau2018invariant}, and quantum mechanics \cite{khaneja}. In this article, we restrict ourself to left invariant systems; however, the theory is well applicable to right invariant systems.} on the matrix Lie group:
\begin{equation}\label{eq:LieDyn}
\begin{aligned}
&\dot{g}(t) = g(t)\xi(t), \\
&\dot{\xi}(t) = f\left(t,\xi(t)\right) + w(t),\\
& y(t) = g(t)\Exp^{v(t)},  
\end{aligned}
\end{equation}
where
\begin{enumerate}[label={(\ref{eq:LieDyn}-\alph*)},leftmargin=*]
\item \m{\left(g(t), \xi(t)\right) \in G \times \mathfrak{g}} is the \textit{state} of the system,
\item \m{\big(v(t),w(t)\big) \in \mathfrak{g}\times \mathfrak{g}} is a vector-valued \textit{white-noise Gaussian process},
\item \m{y(t) \in G} is the \textit{output} of the system,
\item the smooth map \m{\R \times \mathfrak{g}\ni (t,\xi) \mapsto f(t,\xi) \in \mathfrak{g}} defines the \textit{dynamics evolution} in the Lie algebra,
\item and \m{\mathfrak{g} \ni v \mapsto \Exp^v\in G} is the \textit{exponential map}.  
\end{enumerate}
The filtering objective is to design a state estimator of the system \eqref{eq:LieDyn} under the influence of white-noise. We consider that the measurement noise \m{v} and the system noise \m{w} are uncorrelated, zero mean, white noise.   
\begin{assumption}\label{asm:Noise}
The white-noise Gaussian processes \m{v} and \m{w} are mean zero, i.e.,
\[E[v(t)] = 0, \quad  E[w(t)] = 0 \quad \text{for all \;} t,\]
with covariance 
\begin{align*}
 \cov[v(t);v(\tau)] &= \delta(t-\tau)R(t),\\
\cov[w(t);w(\tau)] &= \delta(t-\tau)Q(t),
 \end{align*}
where \m{\delta} is the Dirac delta function,
\m{Q(t)\in L(\mathfrak{g},\mathfrak{g})} is a  symmetric, positive semidefinite operator, and 
\m{R(t)\in L(\mathfrak{g},\mathfrak{g})} is a   symmetric, positive definite operator.   The three signals $v(t)$, $w(t)$ and $(g(t_0), \xi(t_0))$ are mutually independent for all $t$, where $t_0$ is an initial time of interest.
\end{assumption}

\section{Invariant extended Kalman Filter}\label{sec:IEKF}
An objective of a filtering technique is to estimate the state of the system using the system output that is available for measurements. For mechanical systems evolving on Lie groups, the desirable features of a filter are to preserve the configuration manifold and possess a symmetric structure that ensures the filter error dynamics to be independent of the group configuration. A crucial advantage of such invariant filters is that the estimated state evolves on the configuration manifold and error bounds obtained for convergence of the nonlinear filter are independent of the group configuration. 

An invariant extended Kalman filter for the system \eqref{eq:LieDyn} is defined as   
\begin{subequations}\label{eq:KalmanFilter}
\begin{equation}\label{eq:KalmanDyn}
\begin{aligned}
& \dot{h}(t) = h(t)\Big(\eta(t) - K_G(t) \pi_{\mathfrak{g}}\big(y(t)^{-1}h(t)-I\big)\Big),\\ 
& \dot{\eta}(t) = f\left(t,\eta(t)\right) - K_{\mathfrak{g}}(t) \pi_{\mathfrak{g}} \big(y(t)^{-1}h(t)-I\big),
\end{aligned}
\end{equation}
with the Kalman gains
\begin{equation}
\label{eq:KalmanGains}
K(t) \Let \begin{pmatrix}K_G(t) \\ K_{\mathfrak{g}}(t) \end{pmatrix} = \Sigma(t)  C^*R(t)^{-1},
\end{equation}
and the associated differential Riccati equation 
\begin{equation}\label{eq:KalmanRiccati}
\begin{aligned}
\dot{\Sigma}(t) =	&A(t)\Sigma(t) + \Sigma(t) A(t)^* + B Q(t) B^*  \\
			&  - \Sigma(t)C^{*}R(t)^{-1}C\Sigma(t),
\end{aligned} 
\end{equation}
\end{subequations}
where 
\begin{enumerate}[label={(\ref{eq:KalmanFilter}-\alph*)},leftmargin=*]
\item \m{\left(h(t), \eta(t)\right) \in G \times \mathfrak{g}} is the \textit{estimated state} of the system,
\item \m{y(t) \in G} is the \textit{measured output} of the system,
\item the linear maps 
\begin{align*}
&\mathfrak{g}^2 \ni (\alpha,\beta) \mapsto A(t)\cdot (\alpha,\beta) = \begin{pmatrix} -\ad_{\eta(t)}(\alpha) + \beta \\ \frac{\partial f}{\partial \eta} \big(t,\eta(t)\big) \beta \end{pmatrix} \in \mathfrak{g}^2,\\
&\mathfrak{g} \ni \beta \mapsto B \cdot \beta \Let (0,\beta) \in \mathfrak{g}^2,\\
&\mathfrak{g}^2 \ni (\alpha,\beta) \mapsto C \cdot (\alpha,\beta) \Let \alpha \in \mathfrak{g}.
\end{align*}
\end{enumerate}

\begin{remark}
We would like to highlight the fact that the IEKF \eqref{eq:KalmanFilter} is motivated by the conventional EKF on Euclidean spaces. In case we consider the Lie group $ G = \R^n $ with vector addition as a group operation, instead of a matrix Lie group, then the IEKF is identical to the EKF on $ \R^n$.      
\end{remark}

In the next section, we derive a coordinate independent description of the Kalman gains \eqref{eq:KalmanGains} and the associated differential Riccati equation \eqref{eq:KalmanRiccati} (the dynamics of the error covariance \m{\Sigma}) using \textit{tensor algebra} \cite{greub}.

\section{Differential Riccati Equation}\label{sec:DRE}
\subsection{Error dynamics}
Let us define a left invariant error\footnote{Note that under the left action \(G \times \left(G\times G\right) \ni \left(q,(p_1,p_2)\right) \mapsto \Phi_q(p_1,p_2) \Let(qp_1,qp_2) \in G \times G, \) the error $E(g,h) \Let g^{-1}h$ is left invariant, i.e. $E\left(\Phi_q(g,h)\right) = E\left(g,h\right)$ for all $q \in G$.} between the state of the system \eqref{eq:LieDyn} and the state of the IEKF \eqref{eq:KalmanFilter} as
\begin{align}\label{eq:LeftError}
\R \ni t \mapsto \begin{pmatrix} E(t) \\ e(t) \end{pmatrix} \Let \begin{pmatrix} g(t)^{-1} h(t)\\ \eta(t)-\xi(t) \end{pmatrix} \in G \times \mathfrak{g}.
\end{align}
Then the error dynamics is given by
\begin{equation*}
\begin{aligned}
\dot{E}(t)=& g(t)^{-1} \dot{h}(t) - g(t)^{-1} \dot{g}(t) g(t)^{-1} h(t),\\
\dot{e}(t)=& \dot{\eta}(t) - \dot{\xi}(t), 
\end{aligned}
\end{equation*}
which simplifies by substituting the system dynamics \eqref{eq:LieDyn} and the filter dynamics \eqref{eq:KalmanDyn} to
\begin{subequations}\label{eq:ErrorDyn}
\begin{align}
\dot{E}(t)=& - E(t)K_G(t) \pi_{\mathfrak{g}}\big(\Exp^{-v(t)}E(t)-I\big) + [E(t),\eta(t)]\nonumber\\
	   & + e(t)E(t),\label{eq:ErrorDynG} \\
\dot{e}(t)=& - K_{\mathfrak{g}}(t) \pi_{\mathfrak{g}} \big(\Exp^{-v(t)}E(t)-I\big)-f\left(t,\eta(t)-e(t)\right)\nonumber \\
           &+ f\left(t,\eta(t)\right) -w(t),\label{eq:ErrorDynA}
\end{align}
\end{subequations}
where the map 
\[\R^{n \times n} \times \R^{n \times n} \ni (\alpha,\beta) \mapsto [\alpha,\beta] \Let \alpha \beta - \beta \alpha \in \R^{n\times n}\]
 is the matrix commutator. 
\begin{remark}
A key advantage of the IEKF is that the error dynamics \eqref{eq:ErrorDynG} is independent of the state \m{h(t)}, and therefore, the convergence region of the dynamics \eqref{eq:ErrorDynG} is independent of the group configuration \m{h(t)}. In other words, the region of convergence of the dynamics \eqref{eq:ErrorDynG} are uniformly defined due to the left group symmetry of the error \eqref{eq:LeftError}, left invariance properties of the IEKF dynamics \eqref{eq:KalmanDyn} and the plant dynamics \eqref{eq:LieDyn}. 
\end{remark}
In order to linearize the error dynamics \eqref{eq:ErrorDyn} around \m{(I,0) \in G \times \mathfrak{g}}, let us choose a smooth trajectory 
\[ \R \ni \epsilon \mapsto \big(E^{\epsilon}(t), e^{\epsilon}(t),v^{\epsilon}(t),w^{\epsilon}(t) \big) \in  G \times \mathfrak{g}^3 \]
that satisfies the error dynamics \eqref{eq:ErrorDyn} for all $\epsilon \in \R$, in addition to
\[\big(E^0(t), e^0(t), v^0(t),w^0(t) \big) = (I,0,0,0) \in G \times \mathfrak{g}^3\]
and 
\[\left.\frac{d}{d\epsilon}\right|_{\epsilon=0} \big(E^{\epsilon}(t), e^{\epsilon}(t),v^{\epsilon}(t),w^{\epsilon}(t) \big) = \big(\Delta E(t),\Delta e(t), v(t),w(t)\big).\] 
	Therefore, the linearized error dynamics for the system \eqref{eq:ErrorDyn} is given by
\begin{equation}\label{eq:ErrorLinDyn}
\dot{\Delta}(t) = \big(A(t)-K(t)C\big)\Delta(t) - B w(t) + K(t)v(t)
\end{equation}
where
\begin{align*}
&A(t) = \begin{pmatrix} -\ad_{\eta(t)} & I \\ 0 & \frac{\partial f}{\partial \eta} \big(t,\eta(t)\big) \end{pmatrix}, \; B = \begin{pmatrix} 0 \\ I \end{pmatrix}, \;  C = \begin{pmatrix} I & 0 \end{pmatrix}, \\ 
&\Delta(t) \Let \begin{pmatrix}\Delta E(t)\\ \Delta e(t) \end{pmatrix}, \; K(t) \Let \begin{pmatrix} K_G(t) \\ K_{\mathfrak{g}}(t)\end{pmatrix},  \; \text{and\;} \ad_{\eta}(\alpha)  \Let [\eta, \alpha]. 
\end{align*}
In the extended Kalman filter framework, the linear error dynamics is utilized to define the performance measure of the filter \cite{reif2000}. In particular, we design the Kalman gains \eqref{eq:KalmanGains} for the Kalman filter \eqref{eq:KalmanFilter} that minimize the covariance of the process \m{\Delta} whose dynamics is given by \eqref{eq:ErrorLinDyn}.  

\subsection{Minimum variance estimates}
In this section we derive Kalman gains by minimizing the variance of the linearized error process \m{\Delta}, and establish the associated differential Riccati equation. Note that we adopt a generic framework for deriving the differential Riccati equation on the Lie algebra \m{\mathfrak{g}} in which the Lie algebra is not identified with an Euclidean space.    

	Before we proceed with the derivation, let us recapitulate some key notions on isomorphisms related to tensor algebra; refer to \cite{greub} for more detail. 
\begin{fact}[{{\cite[Section 1.28]{greub}}}]\label{fact:Trace}
The tensor product \m{V\otimes W} of two vector spaces $V$ and $W$ is identified with \m{L(V,W)}, the space of linear maps from \m{V} to \m{W}, by the transformation \m{T:V\otimes W \rightarrow L(V,W)} as 
\[ 
T(a\otimes b)x = \ip{a}{x}b \quad \text{for all\;} x \in V,\; a \in V \text{\;and \;} b \in W.
\] 
Further, for any \m{S \in L(W,W), R \in L(V,V), a \in V \text{ and } b \in W,}  
\[
S\circ T(a\otimes b) = T(a\otimes S b) \;\text{ and } \; T(a\otimes b)\circ R = T(R^{*} a\otimes b),   
\]
where $\circ$ is the composition operator. In case \m{W=V}, the trace operator is defined as 
\[
L(V,V) \ni T(a \otimes b) \mapsto \tr\big(T(a\otimes b)\big) = \ip{a}{b}. 
\]
\end{fact}
For the sake of clarity, with a slight abuse of notation, we denote both the linear operator \m{T(a \otimes b)} and the corresponding tensor \m{(a \otimes b)} by \m{(a\otimes b)} as it is very well understood from the context. 

	Let us derive Kalman gains and the associated differential Riccati equation by minimizing variance of the process \m{\Delta}.  The variance minimization problem is defined in optimal control framework as 
\begin{equation}\label{eq:FilterOpt}
\begin{aligned}
\minimize_{K}\,& J (\Delta) \Let E\left[\ip{\Delta(\tau)}{M \Delta(\tau)}\right]\\
\text{subject to}
&\begin{cases}
\text{error dynamics }\eqref{eq:ErrorLinDyn}\; \text{ for } t_0 \leq t \leq \tau,\\
\text{boundary condition }\Delta(t_0) \sim \mathcal{N}(0,\Sigma_0),\\
\tau > t_0 \text{ is fixed.}
\end{cases}
\end{aligned}
\end{equation}
where \m{M, \Sigma_0 \in L(\mathfrak{g}^2,\mathfrak{g}^2)} are symmetric, positive definite, linear  operators on $\mathfrak{g}^2$.

Note that the optimal control problem \eqref{eq:FilterOpt} is in stochastic setting, and it can be reformulated as a deterministic optimal control problem as follows:  
\begin{theorem}\label{thm:CovDyn}
Let 
\[
\R \ni t \mapsto \Sigma(t) \Let E\left[\Delta(t)\otimes \Delta(t)\right] \in L(\mathfrak{g}^2,\mathfrak{g}^2)
\]
be the covariance of the process \m{\Delta}. If Assumption \ref{asm:Noise} holds, then
 the process \m{\Delta} satisfies
\begin{equation}\label{eq:CovDyn}
\begin{aligned}
\dot{\Sigma}(t) = 	&\big(A(t)-K(t)C\big)\Sigma(t) + \Sigma(t) \big(A(t)-K(t)C\big)^* \\
			& + B Q(t) B^* + K(t) R(t) K(t)^*,
\end{aligned} 
\end{equation}
and the cost function \m{J} in  \eqref{eq:FilterOpt} translates to 
\[
J \big(\Sigma(\tau)\big) = \ip{M}{\Sigma(\tau)}.
\]
Furthermore, the optimal control problem \eqref{eq:FilterOpt} is equivalent to the following problem:
\begin{equation}\label{eq:CovOpt}
\begin{aligned}
\minimize_{K}\,&  \ip{{M}}{\Sigma(\tau)}\\
\text{subject to}
&\begin{cases}
\text{dynamics }\eqref{eq:CovDyn}\;\text{ for } t_0 \leq t \leq \tau,\\
\text{boundary condition }\Sigma(t_0) = \Sigma_0,\\
\tau > t_0 \text{ is a fixed given time.}
\end{cases}
\end{aligned}
\end{equation}
\begin{proof}
See Appendix A.
\end{proof}
\end{theorem}

\begin{corollary}
The optimal solution to the optimal control problem \eqref{eq:CovOpt} is given by
\begin{equation}\label{K:gain}
K(t) \Let \Sigma(t)C^*R(t)^{-1},
\end{equation}
where \m{\Sigma(t)} is the solution to the differential Riccati equation
\begin{equation}\label{eq:CovDyn:Riccati}
\begin{aligned}
\dot{\Sigma}(t) =	&A(t)\Sigma(t) + \Sigma(t) A(t)^* + B Q(t) B^*  \\
			&  - \Sigma(t)C^{*}R(t)^{-1}C\Sigma(t),
\end{aligned} 
\end{equation}
with $\Sigma(t_0) = \Sigma_0$.
\begin{proof}
It can be proved in the same way as in  \cite{athans}, so the proof is omitted.
\end{proof}
\end{corollary}

\section{Rigid body attitude control}\label{sec:RigidBody}
Consider the attitude dynamics of a fully actuated rigid body:
\begin{equation}\label{eq:RigidDyn}
\begin{aligned}
&\dot{X}(t) = X(t)\widehat{\Omega}(t)\\
&\dot{\Omega}(t) = \mathbb{I}^{-1}\big(\mathbb{I}\Omega(t)\times \Omega(t) + u(t)\big)+w(t)\\
&Y(t) = X(t)\Exp^{v(t)},
\end{aligned}
\end{equation}
where \m{\big(X(t),\Omega(t)\big) \in \text{SO}(3) \times \R^3} defines the state of the system with \m{X(t)} represents orientation and \m{\Omega(t)} represents angular velocity of the rigid body at time \m{t}, \m{w(t)\in \R^3} is the system noise, \m{v(t)\in \mathfrak{so}(3)} is the measurement noise, \m{u(t)\in\R^3} is the control torque, \m{\mathbb{I}\in\R^{3\times3}} is the moment of inertia of the rigid body, and \m{Y(t)\in \text{SO}(3)} is the output of the system; the {\it hat} map \m{\R^3 \ni x \mapsto \hat{x} \in \mathfrak{so}(3)} is a vector space isomorphism such that $\hat x y  =x \times y$ for all $x,y  \in \R^3$, with \m{\mathfrak{so}(3)} is the Lie algebra of the Lie group \m{\text{SO}(3)}. Further, the inverse of the hat map is denoted by $\mathfrak{so}(3) \ni x \mapsto \pi_{\R^3}(x) \in \R^3.$ 

	We consider the fact that the manifold \m{\text{SO}(3)} is embedded into \m{\R^{3\times 3}} which ensures the existence of unique orthogonal projection map 
\[
\R^{3 \times 3} \ni x \mapsto \pi_{\mathfrak{so}(3)} (x) \in \mathfrak{so}(3)
\] 
with respect to the Euclidean inner product on \m{\R^{3\times 3}}. 
Let \m{\left(Z(t), \omega(t)\right) \in \text{SO}(3) \times \R^3} be the \textit{estimated state} of the system, \m{R(t)}  be the covariance of the output noise \m{v(t)}, and \m{Q(t)}  be the covariance of the system noise \m{w(t)}. 
\subsection{Invariant extended Kalman filter}
	Then the invariant extended Kalman filter for the system is given by
\begin{subequations}\label{eq:KalmanFilterRigidBody}
\begin{equation}
\begin{aligned}
 \dot{Z}(t) = & Z(t)\Big(\hat{\omega}(t) - K_G(t) \pi_{\mathfrak{so}(3)}\big(Y(t)^{-1}Z(t)-I\big)\Big),\\ 
\dot{\omega}(t) = & \mathbb{I}^{-1}\big(\mathbb{I}\omega(t)\times \omega(t) + u(t)\big)\\
&- K_{\mathfrak{g}}(t) \big(\pi_{\R^3} \circ \pi_{\mathfrak{so}(3)}\big)\big(Y(t)^{-1}Z(t)-I\big),
\end{aligned}
\end{equation}
with the Kalman gains
\begin{equation}
\label{eq:KalmanGainsRigidBody}
K(t) \Let \begin{pmatrix}K_G(t) \\ K_{\mathfrak{g}}(t) \end{pmatrix} =\Sigma(t)  C^{*}R^{-1}(t),
\end{equation}
and the covariance trajectory 
\[\R \ni t \mapsto \Sigma(t) \in L\Big(\big(\mathfrak{so}(3),\R^3\big),\big(\mathfrak{so}(3),\R^3\big)\Big)\] 
satisfies the dynamics
\begin{equation}
\label{eq:CovDynRigidBody}
\begin{aligned}
\dot{\Sigma}(t) =	&A(t)\Sigma(t) + \Sigma(t) A(t)^* + B Q(t) B^*  \\
			&  - \Sigma(t)C^{*}R(t)^{-1}C\Sigma(t),
\end{aligned} 
\end{equation}
\end{subequations}
where 
\begin{align*}
&A(t) = \begin{pmatrix} -\ad_{\hat{\omega}(t)} & I \\ 0 & \mathbb{I}^{-1}\big( \widehat{\mathbb{I}\omega(t)} - \widehat{\omega(t)} \mathbb{I}  \big) \end{pmatrix}, \quad B = \begin{pmatrix} 0 \\ I \end{pmatrix}, \\ 
&C = \begin{pmatrix} I & 0 \end{pmatrix}, \quad \text{and } \; \ad_{\eta}(\alpha)  = \eta \alpha - \alpha \eta.
\end{align*}

\subsection{Extended Kalman filter}
The conventional extended Kalman filter (EKF) is implemented in two steps:
\begin{enumerate}
\item First, the rigid body dynamics \eqref{eq:RigidDyn} is embedded into an Euclidean space $\R^{3\times 3} \times \R^3$. Note that the system dynamics \eqref{eq:RigidDyn} is naturally extended from $\text{SO}(3) \times \R^3$ to the Euclidean space $\R^{3 \times 3} \times \R^3 $. 
\item Second, the output map of the embedded dynamics, i.e., 
\begin{equation}\label{eq:RigidOutput}
\mathfrak{so}(3) \ni v \mapsto Y \Let X \Exp^{v(t)} \in \R^{3 \times 3},
\end{equation}
needs to be a local diffeomorphism around $0$. Therefore, we need to approximate the output as
\begin{equation}\label{eq:RigidOutputApp}
\R^{3\times 3} \ni v \mapsto Y \Let X \Exp^{v} \in \R^{3 \times 3},
\end{equation}
such that it is a local diffeomorphism.  
\end{enumerate}
Then the rigid body dynamics \eqref{eq:RigidDyn} with the output \eqref{eq:RigidOutputApp} is considered for designing the extended Kalman filter. 
\begin{remark}
In numerical experiments, the rigid body dynamics \eqref{eq:RigidDyn} is employed to generate the output, and the invariant extended Kalman filter and the conventional extended Kalman filter are employed for estimating the states of the system \eqref{eq:RigidDyn}.  
\end{remark}

\subsection{Numerical results}
We have implemented the invariant extended Kalman filter \eqref{eq:KalmanFilterRigidBody} and a conventional extended Kalman filter for the attitude dynamics. The following parameters have been considered for numerical experiments:
\begin{enumerate}
\item The moment of inertia matrix of the rigid body: \[\mathbb{I}  = \text{diag} (4.250, 4.337,3.664).\]
\item System noise covariance: \m{Q(t) = 2 I \;\text{ for all } t.}
\item Measurement noise covariance: \m{R(t) = 0.3 I \;\text{ for all } t.}
\item Filter initial states: \m{Z(0) = I, \omega(0) = \big(2.1, 0.4, 1.2\big)^\top.}
\item System initial states: \m{X(0) = \Exp^{\text{v}_0}, \Omega(0) = \big(2, 0, 1\big)^\top + \text{w}_0,} where \m{\text{v}_0 \sim \mathcal{N}(0,0.06 I)} and \m{\text{w}_0 \sim \mathcal{N}(0,0.4 I)}.
\item Reference control trajectory: 
	\[\R \ni t\mapsto u(t) \Let \mathbb{I} \dot{\gamma}(t) - \big(\mathbb{I}\gamma(t)\big)\times \gamma(t) \in \R^3,\] 
where 
\[\gamma(t) \Let \big(1+\cos(t),\sin(t)-\sin(t)\cos(t),\cos(t)+\sin^2(t)\big)^\top.\]
\end{enumerate}
\begin{remark}
The moment of inertia of the rigid body is taken from a satellite from the European Student Earth Orbiter (ESEO) \cite{Hegrenas} and other parameters of filters are chosen appropriately to demonstrate the filter performance. 
\end{remark}
We conducted ten thousand Monte Carlo runs for both the filters,  keeping the initial state of the filter unaltered, with the computed the mean square error\[
\text{MSE}(t) \Let \frac{1}{N} \sum_{i=1}^{N}\norm{\left(X^{(i)}(t), \Omega^{(i)}(t)\right) - \left(Z^{(i)}(t),\omega^{(i)}(t)\right)}^2,
\]
where $\left(X^{(i)}(t), \Omega^{(i)}(t)\right)$ is the system state and $\left(Z^{(i)}(t), \omega^{(i)}(t)\right)$ is the filter state for the $i$th Monte Carlo run. In the transient phase (see MSE for $0-\SI{2}{\second}$ in Figure \ref{fig:ErrorComparison}), the IEKF filter converges relatively slower than the EKF because the initial large error might not be well approximated on the Lie algebra. However, as the filters reach steady-state, the MSE of the IEKF is less as compared to the EKF as shown in Figure \ref{fig:ErrorComparison}. These numerical experiments demonstrate that the IEKF performs better than the EKF because it has less steady-state MSE which is very much desirable. The numerical experiments are run on a computer with Intel(R) Core(TM) i7-8700 @ 3.20 GH processor using 16 GB of RAM and on the platform, MATLAB R2018b. The filter dynamics of IEKF and EKF are integrated using the Runge-Kutta 4th order algorithm with a fixed step length $\SI{0.02}{\second}$ for a time duration of $\SI{10}{\second}$, and the computation time for a single Monte Carlo run is shown in Figure \ref{fig:ComTime}. Note that  the covariance dynamics $\dot{\Sigma}$  of the IEKF is evolving on $ L ( \R^{3} \times \R^3, \R^3 \times \R^3) $ whereas the covariance dynamics of the EKF is evolving on $ L(\R^{3 \times 3} \times \R^3, \R^{3\times 3} \times \R^3) $, and therefore, the IEKF is computationally less intensive than the EKF.    

\begin{figure}
\centering
\begin{tikzpicture}[every pin/.style={fill=white}]
        \begin{axis}[yscale=0.55,xmin=0,xmax=10,ymin=0.2,ymax=0.5,xstep=1, xlabel=Time $ (\SI{}{\second}) $, ylabel near ticks, ylabel= MSE, x label style={at={(axis description cs:0.5,-0.4)}}]
            \addplot[skyblue, thick,dotted] table[x=Time, y=InvMSE, col sep=comma]{TrajMSE.txt};
            \addplot[redfaded, thick] table[x=Time, y=ExtMSE, col sep=comma]{TrajMSE.txt};

\coordinate (pt) at (axis cs:7.5,0.26);
\end{axis}

\begin{customlegend}[
legend entries={ 
IEKF,
EKF
},
legend style={at={(6.7,2.95)},font=\footnotesize}] 
    \addlegendimage{mark=ball,ball color=skyblue, draw=white}
    \addlegendimage{mark=ball,ball color=redfaded,draw=white}
\end{customlegend}

\node[pin=177:{%
    \begin{tikzpicture}[baseline,trim axis left,trim axis right,scale=0.95]
    \begin{axis}[tiny,xmin=6,xmax=9,ymin=0.2,ymax=0.25,xstep=1, enlargelimits]
            \addplot[skyblue, thick, dotted] table[x=Time, y=InvMSE, col sep=comma]{TrajMSE.txt};
            \addplot[redfaded, thick] table[x=Time, y=ExtMSE, col sep=comma]{TrajMSE.txt};
    \end{axis}
    \end{tikzpicture}%
}] at (pt) {};
\end{tikzpicture}

	\caption{Mean square error comparison of IEKF (dotted blue) and EKF (solid red).}
	\label{fig:ErrorComparison}
\end{figure}
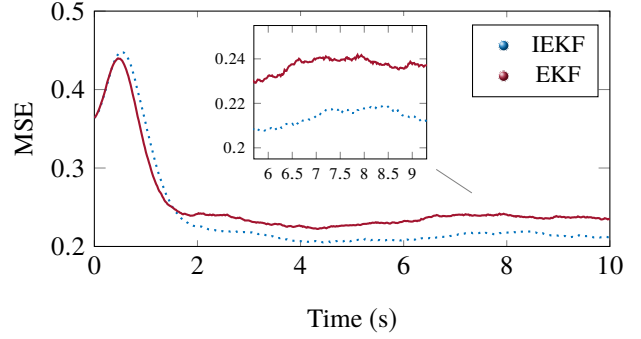 
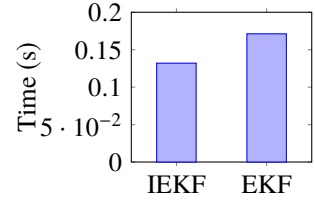
\begin{figure}
\centering
\pgfplotstableread[row sep=\\,col sep=&]{
    Filter & ComTime \\
	0.75& 0.1320\\
    	2.25& 0.1712\\
    }\mydata
\begin{tikzpicture}
        \begin{axis}[yscale=0.35, xscale=0.35, xmin =0, xmax = 3,ymin =0, ymax = 0.2,bar width=15mm,ybar,xtick=data,xticklabels={IEKF,EKF},ylabel shift = 50 pt,ylabel={Time $ (\SI{}{\second}) $}]
            \addplot table[x=Filter, y=ComTime]{\mydata};
        \end{axis}
    \end{tikzpicture}
	\caption{Computation time comparison of IEKF and EKF.}
	\label{fig:ComTime}
\end{figure} 

\begin{ack}                               
Partially supported by the KUSTAR-KAIST Institute, KAIST; by KAIST under grants G04170001 and N11180231; by the Institute for Information \& communications Technology Planning \& Evaluation(IITP) grant funded by the Korea government (MSIT) [2019-0-01396, Development of framework for analyzing, detecting, mitigating of bias in AI model and training data]; and by the ICT R\&D program of MSIP/IITP [2016-0-00563, Research on Adaptive Machine Learning Technology Development for Intelligent Autonomous Digital Companion].
\end{ack}

\bibliography{InvKalmanFilterAutomatica}
\bibliographystyle{siam}
 
\appendix

\section{Proof of Theorem \ref{thm:CovDyn}}
We know that the covariance of the process \m{\Delta} is defined by 
\[
\Sigma(t) \Let E\left[\Delta \otimes \Delta \right] \in L (\mathfrak{g}^2,\mathfrak{g}^2).
\]
Therefore, using the property discussed in Fact \ref{fact:Trace}, we obtain 
\begin{align*}
J(\Delta) & \Let  E\left[\ip{\Delta(\tau)}{M \Delta(\tau)}\right],\\
& = E\left[\tr\big(\Delta(\tau) \otimes M \Delta(\tau) \big) \right],\\
& = E\left[\tr\Big(M\circ \big(\Delta(\tau) \otimes \Delta(\tau) \big) \Big) \right],\\
& = \tr \big(M \Sigma(\tau)\big),\\
& =  \ip{{M}}{\Sigma(\tau)}.
\end{align*} 
We now turn to derive the covariance dynamics as
\begin{align*} 
\dot{\Sigma}= \frac{d}{dt} \Big(E\left[\Delta \otimes \Delta \right]\Big) =	E\left[\dot{\Delta}\otimes \Delta + \Delta\otimes\dot{\Delta}\right],
\end{align*}	
which further simplifies, by substituting the expression of \m{\dot{\Delta}} from \eqref{eq:ErrorLinDyn}, to
\begin{equation}
\begin{aligned}\label{eq:Sigma2}
	\dot{\Sigma}	=	& \big((A-KC)\otimes I\big)\Sigma + \big(I \otimes (A-KC)\big)\Sigma\\
				& - \big(B \otimes I\big)E\left[w\otimes\Delta\right] + \big(K\otimes I\big) E\left[v\otimes \Delta\right]\\
				& - \big(I \otimes B\big)E\left[\Delta\otimes w\right] + \big(I\otimes K\big) E\left[\Delta\otimes v\right]. 
\end{aligned}
\end{equation}
By Lemma \ref{lem:NoiseCov} in Appendix B, the dynamics \eqref{eq:Sigma2} is rewritten as 
\begin{equation}
\begin{aligned}\label{eq:Sigma3}
	\dot{\Sigma}	=	& \big((A-KC)\otimes I\big)\Sigma + \big(I \otimes (A-KC)\big)\Sigma\\
				& + \big(B \otimes B\big)Q + \big(K\otimes K\big)R. 
\end{aligned}
\end{equation}
Subsequently, applying Lemma \ref{lem:Operator} in Appendix B to \eqref{eq:Sigma3} leads to the following covariance dynamics:
\begin{equation}
\dot{\Sigma} = \big(A-KC\big)\Sigma + \Sigma \big(A-KC\big)^* + B Q B^* + K R K^*.
\end{equation}
The equivalence of the optimal problem \eqref{eq:FilterOpt} to \eqref{eq:CovOpt} is now trivial.
This completes the proof of Theorem \ref{thm:CovDyn}.

\section{Auxiliary Lemmas}
\begin{lemma}\label{lem:NoiseCov}
If Assumption \ref{asm:Noise} holds, then the following is true:
\begin{enumerate}
\item \m{E \left[ w(t) \otimes \Delta(t) \right] = -\frac{1}{2}\big(I \otimes B \big) Q(t),} 
\item \m{E \left[ v(t) \otimes \Delta(t) \right] = \frac{1}{2}\big(I \otimes K(t) \big) R(t),} 
\item \m{E \left[ \Delta(t) \otimes w(t) \right] = -\frac{1}{2}\big( B \otimes I \big) Q(t),} 
\item \m{E \left[ \Delta(t) \otimes v(t) \right] = \frac{1}{2}\big( K(t)\otimes I\big) R(t).} 
\end{enumerate}
\end{lemma}

\begin{proof}
Let 
\[
\R^2 \ni (t,\tau) \mapsto \Phi(t,\tau) \in L(\mathfrak{g}^2,\mathfrak{g}^2)
\]
be the state transition operator for the autonomous system
\[
\dot{\Delta}(t) = \Big(A(t)-K(t)C\Big) \Delta(t).
\]
Then, the solution of the error dynamics \eqref{eq:ErrorLinDyn} with the boundary condition \m{\Delta (t_0) = \Delta_0} is given by
\[
\Delta(t) = \Phi(t,t_0)\Delta_0 + \int_{t_0}^{t} \Phi(t,\tau) \big(K(\tau) v(\tau)-B(\tau)w(\tau)\big) d\tau.
\]
We now turn to prove: \m{E \left[ w(t) \otimes \Delta(t) \right] = -\frac{1}{2}\big(I \otimes B \big) Q(t)}. 
\begin{align}\nonumber
& E\left[w(t)\otimes \Delta(t)\right] \\ \nonumber
&\quad=  E \left[w(t)\otimes \Phi(t,t_0) \Delta_0\right]\\\nonumber
&\quad \;\; + E\left[w(t) \otimes \int_{t_0}^{t} \Phi(t,\tau) \big(K(\tau) v(\tau)-B(\tau)w(\tau)\big) d\tau \right]\\\nonumber
&\quad = E\left[(I\otimes\Phi(t,t_0))(w(t)\otimes \Delta_0)\right] \\\nonumber
&\quad \;\; + E\left[\int_{t_0}^{t} \big(I\otimes  \Phi(t,\tau) K(\tau) \big) \big(w(t)\otimes v(\tau)\big) d\tau\right]\\
&\quad \;\; - E\left[\int_{t_0}^{t} \big(I\otimes  \Phi(t,\tau) B(\tau) \big) \big(w(t)\otimes w(\tau)\big) d\tau\right] . \label{eq:CovStW}
\end{align}
Using Assumption \ref{asm:Noise}, \eqref{eq:CovStW} simplifies to 
\begin{align*}
E\left[w(t)\otimes \Delta(t)\right] & = -\int_{t_0}^{t} \big(I\otimes  \Phi(t,\tau) B(\tau) \big) Q(\tau) \delta(t-\tau) d\tau\\
& = -\frac{1}{2}\big(I \otimes B \big) Q(t).
\end{align*}
The remaining assertions of Lemma \ref{lem:NoiseCov} can be proved in an identical manner.  

\end{proof}

\begin{lemma}\label{lem:Operator}
Let \m{T,U,V,W} be the finite dimensional vector spaces, and let \m{X \in L(T,U), Y \in L(V,W)} and \m{Z\in T \otimes V = L(T,V)}. Then 
\[
\big(X\otimes Y\big)  Z = Y \circ Z \circ X^*,
\] 
where $Z$ on the left side is regarded as an element of $T \otimes V$ and $Z$ on the right side as an element of $L(T,V)$.
\end{lemma}
\begin{proof}
Write  \m{Z} as 
\[
Z = \sum_{i,j}  t_i \otimes v_j
\]
for some $t_i$'s in $T$ and $v_j$'s in $V$.
Then using Fact \ref{fact:Trace}, 
\begin{align*}
\big(X \otimes Y \big)  Z &  = \big(X \otimes Y)  \Big(\sum_{i,j}  t_i \otimes v_j\Big)\\
& = \sum_{i,j}  \Big(X(t_i) \otimes Y (v_j) \Big)\\
& = \Big(\sum_{i,j}  t_i \otimes Y(v_j) \Big)\circ X^* \\
& = Y \circ \Big(\sum_{i,j}  t_i \otimes v_j\Big) \circ X^*\\
& = Y \circ Z \circ X^*.
\end{align*}
This proves the assertion. 
\end{proof}

\end{document}